\documentclass[12pt]{amsart}

\usepackage[matrix,arrow,curve,frame]{xy}
\usepackage{amsmath,amsthm,amssymb,enumerate}
\usepackage{latexsym}
\usepackage{amscd}
\usepackage[backref]{hyperref}
\usepackage{euscript}
\usepackage{fullpage}
\usepackage{tikz-cd}

\usepackage{color}

\usepackage[all]{xypic}

\newtheorem{theorem}{Theorem}[section]

\newtheorem{corollary}[theorem]{Corollary}
\newtheorem{proposition}[theorem]{Proposition}

\newtheorem{remark}[theorem]{Remark}

\numberwithin{equation}{section}
\newcommand{\beq}{\begin{equation}}
\newcommand{\eeq}{\end{equation}}

\newcommand{\TT}{\mathbb{T}}
\newcommand{\ZZ}{\mathbb{Z}}

\newcommand{\CC}{\mathbb{C}}

\newcommand\bbT{\mathbb T}

\DeclareMathOperator{\Hom}{Hom}

\DeclareMathOperator{\Tr}{Tr}

\newcommand{\Z}{\mathbb{Z}}

\newcommand{\sU}{{\sf U}}

\renewcommand{\cL}{\mathcal{L}}

\newcommand{\rv}{\mathrm{v}}

\newcommand{\rf}{\mathrm{f}}

\newcommand{\gA}{\mathcal{A}}
\newcommand{\gB}{\mathcal{B}}

\newcommand{\LQ}{L^\mathcal{Q}}
\newcommand{\LDQ}{L^\mathcal{DQ}}
\newcommand{\rmhol}{\mathrm{hol}}

\newcommand{\kappak}{\kappa\!\left(\underline A, \underline {\widehat A}\right)}

\newcommand{\bu}{\bullet}
\newcommand{\bfv}{\mathbf{v}}

\newcommand {\be}{\begin{equation}}
\newcommand {\ee}{\end{equation}}
\newcommand{\h}{\begin{eqnarray*}}
\newcommand{\e}{\end{eqnarray*}}

\begin{document}

%----------------------------------------------------------------------------------------------------------------------------------------------------------------------------------------------------------%

\title
%{Generalized geometry, exotic twisted equivariant cohomology and T-duality}
{T-duality, vertical holonomy line bundles and loop Hori formulae}


%    author one information
\author{Fei Han}
\address{Department of Mathematics,
National University of Singapore, Singapore 119076}
%\curraddr{}
\email{mathanf@nus.edu.sg}
%\thanks{}

%  author two information
 \author{Varghese Mathai}
\address{School of Mathematical Sciences,
University of Adelaide, Adelaide 5005, Australia}
\email{mathai.varghese@adelaide.edu.au}
%\thanks{}

%    \subjclass is required.
\date{}
%\thanks{}

\maketitle

\begin{abstract}

This paper is a step towards realizing T-duality and Hori formulae for loop spaces.
Here we prove T-duality and Hori formulae for winding q-loop spaces, 
which are infinite dimensional subspaces of loop spaces.

\end{abstract}

\tableofcontents

%%%%%%%%%%%%%%%%%%%%%%%%%%%%%%%%%%%%%%%%%%%
\section*{Introduction}
T-duality is a particular equivalence of  type II string theories that is closely related to mirror symmetry \cite{Hori}.
Our long term aim is to prove T-duality and Hori formulae for loop spaces, since T-duality and the Hori
formulae for spacetime should be a shadow of T-duality and Hori formulae for loop space of spacetime.
More precisely, let spacetime $Z$
be a principal $\bbT$-bundle over $X$. Assume that the spacetime $Z$ is endowed with a flux $H$ which is
a representative in the
degree 3 Deligne cohomology of $Z$. Then the T-dual  is a principal
$\widehat \bbT$-bundle over $X$, denoted by $\widehat Z$, with T-dual flux $\widehat H$ on $\widehat Z$. 
Now loop these to get an $L\bbT$ bundle $LZ$ over $LX$, and an $L\widehat\bbT$ bundle $L\widehat Z$ over $LX$.
Averaging $H$ along loops in $Z$,  one obtains the extension of $H$ from $Z$, the space of constant loops, to $LZ$, denoted by $\underline H$ and similarly
let $\underline {\widehat H}$ denote the extension of $ \widehat H$ from the constant loops to $L\widehat Z$.
Then $(LZ, \underline H)$ and $(L\widehat Z, \underline {\widehat H})$ are T-dual, is the {\em loops pace T-dual conjecture}. Let $\Omega(LZ, \cL)^{S^1}$ denote the space of $S^1$-invariant differential forms on the loop space $LZ$ with values in the holonomy 
line bundle $\cL$ arising from the gerbe $H$. As explained in \cite{HM15}, it has a differential
$D_H=\nabla^{\cL} - \iota_K + \underline H$, which is an equivariantly flat superconnection (cf. \cite{MQ}), where $\nabla^{\cL} $
is the connection on $\cL$ arising from the gerbe with connection $H$ and $K$ is the rotation vector field on $LZ$.
Then the putative Hori formula gives a degree shifting isomorphism and chain map between the complexes (with $u$ being a degree 2 indeterminate),
$$
(\Omega^\bullet(LZ, \cL)^{S^1}[[u, u^{-1}]], \nabla^{\cL} - u\iota_K +u^{-1} \underline H)\quad \text{and}\quad  (\Omega^\bullet(L\widehat Z, \widehat\cL)^{S^1}[[u, u^{-1}]], \nabla^{\widehat \cL} - u\iota_{\widehat K} +u^{-1} \underline {\widehat H})
$$
which is the more precise form of our loop space T-dual conjecture. 

We call the cohomology of the complex 
$$(\Omega^\bullet(LZ, \cL)^{S^1}[[u, u^{-1}]], \nabla^{\cL} - u\iota_K +u^{-1} \underline H)$$
the {\bf completed exotic twisted $S^1$-equivariant cohomology} of $LZ$. In \cite{HM15}, we proved that it is localized to the twisted cohomology $H^\bullet(Z, H)[[u, u^{-1}]]$ of $Z$, the constant loop space. Twisted cohomology first appeared in string theory in the work of Rohm-Witten \cite{RW} and was later studied in detail in \cite{BCMMS,MS}. Therefore the T-duality for loop spaces holds on the level of completed exotic twisted $S^1$-equivariant cohomology after being localized to constant loops. As a result we did not get a loop space version of Hori formula in that paper.

In \cite{HM18}, motivated by the vertical loop spaces of circle bundles and their homotopy equivalent components, we were able to prove T-duality for exotic differential forms on spacetime (see the brief summary in Theorem \ref{main-exotic1} and Theorem \ref{main-exotic2}). The results in \cite{HM18} generalized the T-duality results of \cite{BEM04a, BEM04b} on invariant differential forms. See also \cite{BS} for 
a topological approach to some of the results in  \cite{BEM04a, BEM04b}. The Hori formulae in \cite{HM18} is still on the spacetimes themselves rather than on their loop spaces. 

In \cite{HM22}, we studied the T-duality and Hori formulae for small double loop spaces. Differential forms on 
small double loop spaces can be represented by differerential form valued Jacobi forms on spacetimes, and we obtain the desired
result on these forms as well as proving modularity results in the presence of twisted string structures.

\cite{HM15, HM18} both represent our attempts towards the loop space perspectives of T-duality. Although they are rooted in the study of loop spaces, the results were all stated in terms of the data on $Z$ and $\hat Z$, the constant loop spaces, which are finite dimensional. In this paper, we will prove T-duality results in infinite dimension for the first time by constructing the loop Hori formulae for the {\em winding q-loop spaces} (see \eqref{q-loopspace}), which are infinite dimensional subspaces of loop spaces. 

More precisely, let $(Z, A, H)$ and $(\widehat Z, \widehat  A, \widehat  H)$ be a T-dual pair, with $Z, \widehat Z$ being circle bundles over a same base $X$, $A$ being a connection and $H$ being a flux on $Z$, and $\widehat A, \widehat H$ being the dual data on $\widehat Z$.  Denote by $\LDQ X$ the space of loops in $X$ on which the holonomies of $Z$ and $\widehat Z$ are trivial. Generically this is an infinitely dimensional subspace of the loop space $LX$. Denote by $\LDQ Z$ the space of loops $x(s)$ in $Z$ such that their projection onto $X$  lie in $\LDQ X$ and the winding $A(\dot x(s))\in \ZZ, \forall s\in [0, 1]$. We call $\LDQ Z$ the {\bf winding quantized loop space}, or {\bf winding q-loop space} in short.  It decomposes according to the integer values $A(\dot x(s))$ as $\LDQ Z=\coprod_{m\in \Z} \LDQ_mZ$. Dually we have $\LDQ \widehat Z=\coprod_{n\in \Z} \LDQ_n\widehat Z$. We discover that on the winding q-loop space $\LDQ Z$, there is a {\bf vertical holonomy line bundle} with connection $(\cL_H^\rv, \nabla^{\cL^\rv_H})$ arising from the flux $H$, and the superconnetion $\nabla^{\cL^\rv_H}-\iota_{K^\rv}+\underline H$ is equivariantly flat, where $K^\rv$ is the vertical projection of $K$. This allows us to have a complex for each $k\in \ZZ$,
$$(\gA^{\rv, \bullet}(\LDQ_kZ), \nabla^{\cL^\rv_H}-u\iota_{K^\rv}+u^{-1}\underline H),$$
where $\gA^{\rv, \bullet}(\LDQ_kZ):=\Omega^\bullet(\LDQ_kZ,\cL^\rv_H)^\TT[[u, u^{-1}]]$. On the other hand, on $\LDQ_kZ$, there are also complexes arising from momentums,
$$(\gB_{-j}^{\bullet}(\LDQ_k Z), d+u^{-1}\underline{H}),$$
where $\gB_{-j}^{\bullet}(\LDQ_k Z):=\{\omega\in \Omega^{\bullet}(\LDQ_kZ)[[u, u^{-1}]]|\, L_{\bfv}\omega=-j\omega \},  j\in \ZZ$ and $\bfv$ is the vertical vector field along loops in $Z$ arising from the $\TT$-action $Z$. One has dual complexes on the dual side $\LDQ \widehat Z$. We will show that there are loop Hori maps $\tau_{m,n}, \sigma_{m, n}, \widehat \sigma_{n, m},  \widehat \tau_{n, m}$ 
\[ 
\begin{tikzcd}
\gB_{-n}^{\bullet}(\LDQ_m Z)\oplus\gA^{\rv, \bullet}(\LDQ_mZ)\ar[rr, shift left]{rr}{\tau_{m, n}\oplus \sigma_{m, n}} & & \gA^{\rv, \bullet+1}(\LDQ_n\widehat Z)\oplus\gB_{-m}^{\bullet+1}(\LDQ_n \widehat Z)
\ar[ll, shift left]{ll}{\widehat \sigma_{n, m}\oplus \widehat \tau_{n, m}}\end{tikzcd}
\] 
such that 
$$\widehat{\sigma}_{n, m}\circ\tau_{m,n}=-\mathrm{Id}, \  \tau_{m,n}\circ \widehat{\sigma}_{n, m}=-\mathrm{Id} $$
and 
$$\sigma_{m, n}\circ\widehat\tau_{n,m}=-\mathrm{Id}, \  \widehat\tau_{n,m}\circ \sigma_{m, n}=-\mathrm{Id}. $$

We are able to obtain the above loop Hori formulae by observing that $\LDQ_mZ$ and $\LDQ_n\widehat Z$ are principal circle bundles over $ \LDQ X$ and constructing the fiber product, $\cL(m , n)$, which we call
 {\bf $(m, n)$-loop correspondence space} \begin{equation*} \label{eqn:loopcorrespondence}
\xymatrix @=6pc @ur { (\LDQ_m Z, [\underline H])\ar[d]_{L\pi} &
(\cL(m, n), [\underline H]=[\underline{\widehat H}])\ar[d]_{\widehat p_n} \ar[l]^{p_m} \\ \LDQ X & (\LDQ_n \widehat Z, [\underline{\widehat H}])\ar[l]^{L\widehat \pi}}
\end{equation*}
and then use the tool {\bf twisted integration along the fiber} introduced in \cite{HM18}. 
This picture gets us closer to our ultimate aim of T-duality for loop spaces.

The paper is organized as follows. In Section \ref{ReviewT}, we give a brief review of T-duality. In Section \ref{loopofcirclebundle}, we first review the theory of exotic twisted equivariant cohomology on loop spaces introduced in \cite{HM15} and then specialize to loop spaces of princical bundles by introducing the various construction including the winding q-loop spaces, the $(m, n)$-loop correspondence spaces, the vertical holonomy line bundles. Then in Section \ref{secloopHori}, we construct the loop Hori formulae after reviewing the twisted integration along the fiber. In the bulk of this section, we also discuss the relation of the loop Hori formulae to the exotic Hori formulae introduced in \cite{HM18}. 

\bigskip

\noindent{\em Acknowledgements.} Fei Han was partially supported by the grant AcRF R-146-000-218-112 from National University of Singapore. Varghese Mathai was supported by funding from the Australian Research Council, through the Australian Laureate Fellowship FL170100020.

\bigskip

%%%%%%%%%%%%%%%%%%%%%%%%%%%%%%%%%%%%%%%%%%%
\section{Review of T-duality} \label{ReviewT}

In \cite{BEM04a, BEM04b}, spacetime $Z$ was compactified in one direction.
More precisely, $Z$
is a principal $\bbT$-bundle over $X$

\begin{equation}\label{eqn:MVBx}
\begin{CD}
\bbT @>>> Z \\
&& @V\pi VV \\
&& X \end{CD}
\end{equation}
 classified up to isomorphism by its first Chern class
{ $c_1(Z)\in H^2(X,\ZZ)$}. Assume that the spacetime $Z$ is endowed with an $H$-flux which is
a representative in the
degree 3 Deligne cohomology of $Z$, that is
$H\in\Omega^3(Z)$ with integral periods (for simplicity, we drop factors of $\frac{1}{2\pi i}$),
together with
the following data. Consider a local trivialization $U_\alpha \times \TT$ of $Z\to X$, where
$\{U_\alpha\}$ is a good cover of $X$. Let $H_\alpha = H\Big|_{ U_\alpha \times \TT}
= d B_\alpha$, where $B_\alpha \in \Omega^2(U_\alpha \times \TT)$ and finally, 
\be \label{difference} B_\beta -B_\alpha = F_{\alpha\beta}
\in \Omega^1(U_{\alpha\beta} \times \TT).\ee
 Then the choice of $H$-flux entails that we are given a local trivialization
 as above and locally defined 2-forms $B_\alpha$ on it, together with closed 2-forms $F_{\alpha\beta}$ defined on double overlaps,  that is, the Deligne class $(H, B_\alpha, F_{\alpha\beta})$. Also the first Chern class
 of $Z\to X$
  is represented in the integral cohomology by $(F, A_\alpha)$ where
$\{A_\alpha\}$ is a connection 1-form on $Z\to X$ and $F = dA_\alpha$ is the curvature 2-form of $\{A_\alpha\}$.

The {  {T-dual}}  is another principal
circle bundle over $X$, denoted by $\widehat Z$,
  {}
\begin{equation}\label{eqn:MVBy}
\begin{CD}
\widehat \bbT @>>> \widehat Z \\
&& @V\widehat \pi VV     \\
&& X \end{CD}
\end{equation}
To define it, we see that $\pi_* (H_\alpha) = d \pi_*(B_\alpha) = -d {\widehat A}_\alpha$,
 so that $\{{\widehat A}_\alpha\}$ is a connection 1-form, whose curvature $ d {\widehat A}_\alpha = \widehat F_\alpha =  \pi_*(H_\alpha)$
 that is, $\widehat F = \pi_* H$. So let $\widehat Z$ denote the principal
$\widehat \bbT$-bundle over $X$ whose first Chern class is  $\,\, c_1(\widehat Z) = [\pi_* H, \pi_*(B_\alpha)] \in H^2(X; \ZZ) $.

The Gysin
sequence \cite{BT} for $Z$ enables us to define a T-dual $H$-flux
$[\widehat H]\in H^3(\widehat Z,\ZZ)$, satisfying
\begin{equation} \label{eqn:MVBc}
c_1(Z) = \widehat \pi_* \widehat H \,,
\end{equation}
where $\pi_* $
and similarly $\widehat\pi_*$, denote the pushforward maps.
Note that $ \widehat H$ is not fixed by this data, since adding any integer
degree 3 cohomology class on $X$ that is pulled back to $\widehat Z$
also satisfies the requirements. However, $ \widehat H$ is
determined uniquely (up to cohomology) upon imposing
the condition $[H]=[\widehat H]$ on the correspondence space $Z\times_X \widehat Z$
as will be explained now.

The {\em correspondence space} (sometimes called the doubled space) is defined as
$$
Z\times_X  \widehat Z = \{(x, \widehat x) \in Z \times \widehat Z: \pi(x)=\widehat\pi(\widehat x)\}.
$$
Then we have the following commutative diagram,
\begin{equation} \label{eqn:correspondence}
\xymatrix @=6pc @ur { (Z, [H]) \ar[d]_{\pi} &
(Z\times_X  \widehat Z, [H]=[\widehat H]) \ar[d]_{\widehat p} \ar[l]^{p} \\ X & (\widehat Z, [\widehat H])\ar[l]^{\widehat \pi}}
\end{equation}
By requiring that
$$
p^*[H]={\widehat p}^*[\widehat H] \in H^3(Z\times_X  \widehat Z, \ZZ),
$$
determines $[\widehat H] \in H^3(\widehat Z, \ZZ)$  uniquely, via an application of the Gysin sequence.

An alternate way to see this is explained below.

Let $(H, B_\alpha, F_{\alpha\beta})$ denote the Deligne class of the closed integral 3-form $H$. Without loss of generality, we can assume that $H$ is $\TT$-invariant. We also choose a connection 1-form $A$ on $Z$.
Let $v$ denote the vector field generating the $\TT$-action on $Z$. Then define $\widehat A_\alpha = -\imath_v B_\alpha$ on the chart $U_\alpha$ and
the connection 1-form $\widehat A= \widehat A_\alpha +d\widehat\theta_\alpha$
on the chart $U_\alpha\times  \widehat \TT$. In this way we get a T-dual circle bundle
$\widehat Z \to X$ with connection 1-form $\widehat A$ and transition functions $\widehat h_{\alpha\beta}$ such that $\widehat A_\alpha -\widehat A_\beta=-d\ln\widehat h_{\alpha\beta}.$ Let $(L_{\alpha\beta}, \nabla^{L_{\alpha\beta}})$ be a line bundle with $\TT$-invariant connection over $\pi^{-1}(U_\alpha\cap U_{\beta})$ such that $(\nabla^{L_{\alpha\beta}})^2=F_{\alpha\beta}$. Then $(H, B_\alpha, F_{\alpha\beta}, (L_{\alpha\beta}, \nabla^{L_{\alpha\beta}}))$ is a gerbe with connection on $Z$.  We can further require that  $\mathrm{hol}^{\rf}(\nabla^{L_{\alpha\beta}})=\widehat h_{\alpha\beta},$ where $\mathrm{hol}^{\rf}(\nabla^{L_{\alpha\beta}})$ denotes the holonomy of the line bundle with connection $(L_{\alpha\beta}, \nabla^{L_{\alpha\beta}})$ on the fiber circles over $\pi^{-1}(U_\alpha\cap U_{\beta})$.  Note that this requirement can always be satisfied. Actually suppose $\nabla^{L_{\alpha\beta}}$ is any $\TT$-invariant connection such that $\nabla^{L_{\alpha\beta}}=F_{\alpha\beta}$. One has
\be
d\ln\mathrm{hol}^{\rf}(\nabla^{L_{\alpha\beta}})=-\iota_v F_{\alpha\beta}
=\iota_v B_\alpha-\iota_vB_\beta=\widehat A_\beta-\widehat A_\beta=d\ln \widehat h_{\alpha\beta}.
\ee
Then $\mathrm{hol}^{\rf}(\nabla^{L_{\alpha\beta}})$ and  $\widehat h_{\alpha\beta}$ only differ by a multiple of a constant and one can always adjust $\nabla^{L_{\alpha\beta}}$ by adding a $\TT$-invariant  1-form to make its fiberwise holonomy equal to $\widehat h_{\alpha\beta}$.

 Consider
$$
\Omega = H - A\wedge F_{\widehat A}
$$
where  $F_{\widehat A} = d {\widehat A}$ and $F_{A} = d {A}$ are the curvatures of $A$
and $\widehat A$ respectively. One checks that the contraction $i_v\Omega=0$ and
the Lie derivative $L_v\Omega=0$ so that $\Omega$ is a basic 3-form on $Z$, that is
$\Omega$ comes from the base $X$. Set
$$
\widehat H = F_A\wedge {\widehat A} + \Omega.
$$
This defines the T-dual flux 3-form on $\widehat Z$. One verifies that $\widehat H$ is a closed 3-form on $\widehat Z$.
It follows that on the correspondence space, one has as desired,
\begin{equation} \label{HhatH}
\widehat H = H + d (A\wedge \widehat A ).
\end{equation}

Our next goal is to determine the T-dual curving or B-field.
The Buscher rules imply that on the open sets $U_\alpha \times \TT\times \widehat \TT$ of the
correspondence space $Z\times_X \widehat Z$, one has
\begin{equation}\label{Busch}
\widehat B_\alpha = B_\alpha + A\wedge \widehat A - d\theta_\alpha \wedge d\widehat \theta _\alpha\,,
\end{equation}
Note that
\begin{equation}
\imath_v \widehat B_\alpha = \imath_v
\left( B_\alpha + A\wedge \widehat A - d\theta_\alpha \wedge d\widehat \theta _\alpha\right) =
-\widehat A_\alpha + \widehat A - d\widehat \theta_\alpha = 0
\end{equation}
so that $\widehat B_\alpha$ is indeed a 2-form on $\widehat Z$ and not just on the correspondence
space. Obviously, $d \widehat B_\alpha = \widehat H$. Let $\widehat F_{\alpha\beta}=\widehat B_\beta-\widehat B_\alpha$. Then we obtain the Deligne class $(\widehat H, \widehat B_\alpha, \widehat F_{\alpha\beta})$ on $\widehat Z$. From (\ref{Busch}), it is not hard to see that $-\imath_{\widehat v} \widehat B_\alpha=A_\alpha$. Let $(\widehat L_{\alpha\beta}, \nabla^{\widehat L_{\alpha\beta}})$ be line bundles with connections over $\widehat\pi^{-1}(U_\alpha\cap U_{\beta})$ such that $(\widehat \nabla^{L_{\alpha\beta}})^2=\widehat F_{\alpha\beta}$ Then one arrives at the complete
T-dual gerbe with connection, $(\widehat H, \widehat B_\alpha, \widehat F_{\alpha\beta}, (\widehat L_{\alpha\beta}, \nabla^{\widehat L_{\alpha\beta}}))$
(cf. \cite{BMPR}). We can also further require that $\mathrm{hol}^{\rf}(\nabla^{\widehat L_{\alpha\beta}})=h_{\alpha\beta},$ where $\mathrm{hol}^{\rf}(\nabla^{\widehat L_{\alpha\beta}})$ is the holonomy of the line bundles with connection $(\widehat L_{\alpha\beta}, \nabla^{\widehat L_{\alpha\beta}})$ on the fiber circles over $\widehat \pi^{-1}(U_\alpha\cap U_{\beta})$ and $h_{\alpha\beta}$ is the transition function of $Z$.

Define the Riemannian metrics on $Z$ and $\widehat Z$ respectively by
$$
g=\pi^*g_X+R^2\, A\odot A,\qquad \widehat g=\widehat\pi^*g_X+1/{R^2} \,\widehat A\odot\widehat A,
$$
where $g_X$ is a Riemannian metric on $X$.
 Then $g$ is $\TT$-invariant and the length of each circle fibre is $R$; $\widehat g$
 is $\widehat\TT$-invariant and the length of each circle fibre is $1/R$.

The rules for transforming the Ramond-Ramond (RR) fields can be encoded in the (\cite{BEM04a, BEM04b}) generalization of
{\em Hori's formula}
  {}
\begin{equation} \label{eqn:Hori}
T_*G =  \int^{\bbT} e^{ -A \wedge \widehat A }\ G \,,
\end{equation}
 where $G \in \Omega^\bullet(Z)^\bbT$ is the total RR field-strength,
\begin{center}
$G\in\Omega^{even}(Z)^\bbT \quad$ for {   { Type IIA}};\\
$G\in\Omega^{odd}(Z)^\bbT \quad$ for {   { Type IIB}},\\
\end{center}
and where the right hand side of equation \eqref{eqn:Hori} is an invariant differential form on $Z\times_X\widehat Z$, and
the integration is along the $\bbT$-fiber of $Z$.

Recall that the twisted cohomology
is defined as the cohomology of the complex
\be H^\bullet(Z, H) = H^\bullet(\Omega^\bullet(Z), d_H=d+ H\wedge).\ee
By the identity \eqref{eqn:Hori}, $T_*$ maps $d_H$-closed forms $G$ to $d_{\widehat
H}$-closed forms $T_*G$.
 So T-duality $T_*$  induces a map on twisted cohomologies,
$$
T : H^\bullet(Z, H) \to H^{\bullet +1}(\widehat Z, \widehat H).
$$

%%%%%%%%%%%%%%%%%%%%%%%%%%%%%%%%%%%%%%%%%%%%%%%%%%%%%%%%%%%%%%%%%%%%%%%%%%%%%%%%%%%%%%
\section{Loop spaces and vertical holonomy line bundles}\label{loopofcirclebundle}

In this section, we first briefly recap the holonomy line bundles over loop spaces and the exotic twisted equivariant cohomology introduced in \cite{HM15}. Then we specialize to the loop spaces of principal circle bundles with connections that appear in the T-duality settings.  

\subsection{Review of exotic twisted equivariant cohomology on loop spaces} \label{reviewLM} Let $M$ be a smooth manifold and $\{\sU_\alpha\}$ be an open cover of $M$. To distinguish from the special case of principal circle bundles to be discussed  later in the context of T-duality, we use serif font type to denote the open sets in $M$. When the open cover $\{\sU_\alpha\}$ has some nice property,  $\{L\sU_\alpha\}$ can be an open cover of $LM$. For instance, if $\{\sU_\alpha\}$ is a maximal open cover of $M$ with the property that $H^i(\sU_{\alpha})=0$ for $i=2, 3$ and $\forall \alpha$, then $\{L\sU_\alpha\}$ is an open cover of $LM$. In fact, let $x:S^1\to M$ be a smooth loop in $M$ and $\sU_x$ a tubular neighbourhood of $x$ in $M$. Then $\{L\sU_x, x\in LM\}$ covers $LM$.

Let 
\beq\label{eqn:trans}
\tau: \Omega^\bullet(\sU_{\alpha_I} ) \longrightarrow \Omega^{\bullet-1}(L\sU_{\alpha_I} )
\eeq
be the transgression map 
\beq
\tau(\xi_I) = \int_{S^1} ev^*(\xi_I), \qquad \xi_I \in \Omega^\bullet(\sU_{\alpha_I} ).
\eeq
Here $ev$ is the evaluation map
\beq
ev: S^1 \times LM \to M: (t, x)\mapsto x(t).
\eeq

Let $\omega \in \Omega^i(M)$. Define $\widetilde\omega_s \in \Omega^i(LM)$ for $s\in [0,1]$ by
\beq
\widetilde \omega_s(X_1, \ldots, X_i)(x) = \omega(X_1\big|_{x(s)}, \ldots, X_i\big|_{x(s)})
\eeq
for $x\in LM$ and $X_1, \ldots, X_i$ vector fields on $LM$ defined near $x$. Then one checks that
$d \widetilde\omega_s = {\widetilde{d\omega}}_s$. Constructions such as this on loop space were considered in \cite{B85}.

The $i$-form, averaging $\omega$ on the loop space, 
\beq
\underline\omega:=\int_{0}^1  \widetilde\omega_s ds \in \Omega^i(LM)
\eeq
is $S^1$-invariant, that is, $L_K\left(\underline\omega\right) = 0$, where $K$ is the vector field on $LM$ generating rotation of loops. Moreover it is not hard to see that
$$d\underline\omega=\underline{d\omega}, \ \ \tau(\omega) = \iota_K\underline\omega.$$ We call $\underline \omega$ the average of $\omega$. 

Let $(H, B_\alpha, F_{\alpha\beta}, (L_{\alpha\beta}, \nabla^{L_{\alpha\beta}}))$ be a gerbe with connection on $M$, where 
$(H, B_\alpha, F_{\alpha\beta})$ denotes the Deligne class of the closed integral 3-form $H$ and $ (L_{\alpha\beta}, \nabla^{L_{\alpha\beta}})$ denotes the line bundle with connection on double overlaps $\sU_\alpha\cap \sU_\beta$ that determines the gerbe.

On the triple intersection $\sU_\alpha\cap \sU_\beta\cap \sU_\gamma$, there is a trivilization
\be \label{triv} (L_{\alpha\beta}, \nabla^{L_{\alpha\beta}})\otimes (L_{\beta\gamma}, \nabla^{L_{\beta\gamma}})\otimes (L_{\gamma\alpha}, \nabla^{L_{\gamma\alpha}})\simeq (\CC, d). \ee

For any loop $x\in L\sU_\alpha\cap L\sU_\beta$, i.e. $x: S^1\to \sU_\alpha\cap \sU_\beta$, consider the parallel transport equation of the line bundle with connection  $(L_{\alpha\beta},\nabla^{L_{\alpha\beta}})$,
\be \label{ptran}  \nabla_{\dot{x}(s)} \tau^0_s=0,\ \ \ \ \ \ \tau^0_0=\mathrm{Id}, \ee
where $\tau^0_s\in \Hom(L_{\alpha\beta}|_{x(0)}, L_{\alpha\beta}|_{x(s)})$ and $\dot{x}(s)$ is the tangent vector of the loop at time $s$. The holonomy of this parallel transport gives a smooth function 
\be g_{\alpha\beta}=\Tr \tau^0_1\ee
on $L\sU_\alpha\cap L\sU_\beta$.

It can seen from (\ref{triv}) that the following equality holds,
\be g_{\alpha\beta}g_{\beta\gamma}g_{\gamma\alpha}=1. \ee
The holonomy line bundle of this gerbe is defined as 
\be \cL_H:= \left(\coprod_{\alpha\in I} \{\alpha\}\times L\sU_\alpha\times \CC\right)\bigg/\sim,  \ee
where
\be (\beta, x, w)\sim  (\alpha, x, g_{\alpha\beta}(x)\cdot w), \ \ \forall \alpha, \beta\in I,\,  x\in L(\sU_\alpha\cap \sU_\beta), w\in \CC. \ee
Denote by $s_\alpha=(x, 1)$ the local section of $\cL_H$ on $L\sU_\alpha$.

\begin{proposition}[\protect Brylinski, Section 6.1 in \cite{Brylinski}] \label{gauge} The system of one forms $\{-\iota_{K}\underline {B_\alpha}\}$ obey the following gauge transformation laws,
\be d\ln g^{-1}_{\alpha\beta}=-\iota_{K}\underline {B_\alpha}-(-\iota_{K}\underline {B_\beta}). \ee

\end{proposition}

This proposition actually allows one to equip a natural connection on the holonomy line bundle $\cL_H$. 
Let $\nabla^{\cL_H}$ be the connection on  $\cL_H$ such that on $L\sU_\alpha$ under the trivialization by the basis $s_\alpha$, the connection one form is  $-\iota_{K}\underline B_\alpha$.

Consider $\Omega^\bullet(LM, \cL_H)$, which is the space of differential forms on loop space $LM$ with values in
the holonomy line bundle $\cL_H \to LM$. Define $D_{\underline H}=\nabla^{\cL_H}-\iota_{K}+\underline H$. In \cite{HM15}, we found that the supperconnection $D_{\underline H}$ is $S^1$-equvariantly flat. 

\begin{theorem}[\cite{HM15}] \label{flat} $(D_{\underline H})^2=0$ on $\Omega^\bullet(LM, \cL_H)^{S^1}$.
\end{theorem}

This $S^1$-equvariant flatness allows us to introduce a cohomology theory on the loop space. The {\bf completed periodic exotic twisted $S^1$-equivariant cohomology} 
$$h^\bu_\TT(LM, \nabla^{\cL_H}:H)$$ introduced in \cite{HM15} is defined to be the cohomology of the complex
\be (\Omega^\bu(LM, \cL_H)^{S^1}[[u, u^{-1}]], \nabla^{\cL_H}-u\iota_{K}+u^{-1}\underline H), \ee
where $\mathrm{degree}(u)=2$.  

In \cite{HM15}, we established the following localization theorem, which localizes the completed periodic exotic twisted $S^1$-equivariant cohomology to the twisted cohomology of fixed point submanifold of $LM$, i.e. $M$. 
\begin{theorem}[\protect \cite{HM15}] \label{local} 
\be
h^\bu_\TT(LM, \nabla^{\cL_H}:H) \cong H^\bu(\Omega(M)[[u, u^{-1}]], d+u^{-1}H) \cong H^\bu(M, H)[[u, u^{-1}]].
\ee
\end{theorem}

$\, $

\subsection{Loop spaces of principal circle bundles with connections and winding q-loop spaces} \label{loopprincipal} Now consider the situation in the T-duality picture, i.e. a principal circle bundle 
\begin{equation}\label{eqn:MVBx}
\begin{CD}
\bbT @>>> Z \\
&& @V\pi VV \\
&& X \end{CD}
\end{equation}
with a $\TT$-invariant connection 1-form $A$ and a background $\TT$-invariant 3-form flux $H$. We are now interested in the loop space $LZ$. Looping the bundle, we get a principal $L\TT$-bundle 
\begin{equation}
\begin{CD}
L\bbT @>>> LZ \\
&& @VL\pi VV \\
&& LX \end{CD}
\end{equation}

On $LZ$ it is clear that there are two circle actions: one comes from the circle action on the fibers of $Z$ and the other one comes from rotating loops. {\it To distinguish these two circle actions, we call them $\TT$-action and $S^1$-action respectively.} 
Let $v$ denote the vector field generating the $\TT$-action on $Z$. Denote by $\mathbf{v}$ the vector field generating by the the $\TT$-action on $LZ$. Since the $\TT$-action and the $S^1$-action are obviously commutative, we have $[K, \bfv]=0$, where $K$ still denotes the rotating vector fields along loops.  Later we will use $\hat v$, $\widehat K$ and $\hat {\mathbf{v}}$ to denote the similar vector fields on the dual side $\widehat Z$.

Let $\{U_\alpha\}$ be a Brylinski cover of $X$, i.e. $\{U_\alpha\}$ is a maximal open cover of $X$ with the property that $H^i(U_{\alpha_I})=0$ for $i=2, 3$ where $U_{\alpha_I} = \bigcap_{i\in I} U_{\alpha_i}, \,$ $|I|<\infty.$ Then $\{LU_\alpha \}$ gives an open cover of the loop space $LX$. Since $H^2(U_\alpha)=0$ and a principal circle bundle is topologically determined its first Chern class, one has 
\be \pi^{-1}(U_\alpha) \overset{\phi_\alpha}{\simeq} U_\alpha\times \TT.\ee
If $x: S^1\to Z$ is a smooth loop in $Z$, then $\pi\circ x: S^1\to X$ is a smooth loop in $X$. Suppose $\pi\circ x$ lies in some $U_\alpha$, then $x$ lies in $\pi^{-1}(U_\alpha)$. Therefore we see that $\{L\pi^{-1}(U_\alpha)\}$ is an open cover of $LZ$. 

To define the loop version of Hori formula, we restrict ourselves to a smaller loop space rather than $LZ$. Denote 
\be\label{q-loopspace} \LQ Z:=\left\{x: S^1\to Z \left|\, A(\dot x(s))\in \Z, \forall s\in [0, 1]\right.\right\}.\ee
We call $\LQ Z$  the {\bf winding quantized loop space} or {\bf wingding q-loop space} in short. 
\begin{remark}  $A(K(s))\in \Z$ actually should be  $A(\dot x(s))\in 2\pi i\Z$. For simplicity,  in the article we drop the factor $2\pi i$ for similar places. 
\end{remark}
\begin{remark}  Here we abuse the use of this terminology when there is a single circle bundle rather than there is a T-dual pair of circle bundles, where we also also use this terminology and denote it by $\LDQ Z$. Compare (\ref{q-loopspace}), (\ref{q-loopspaceX}), (\ref{dq-loopspaceX}), (\ref{dq-loopspaceZ}) and (\ref{dq-loopspacehatZ}).
\end{remark}

Decompose $\LQ Z$ into $\Z$-many disjoint components
\be \LQ Z=\coprod_{m\in \Z} \LQ_mZ,\ee
where 
\be \LQ_mZ:=\left\{x: S^1\to Z \left|\, A(K(s))=m, \forall s\in [0, 1]\right. \right\}.\ee

\begin{theorem}\label{hol} If $x\in \LQ Z$, then $\mathrm{hol}^Z(\pi\circ x)$, the holonomy of the circle bundle $Z$ along the projection loop $\pi\circ x$, is trivial. 
\end{theorem}
\begin{proof}
For $x: S^1\to Z$, suppose $\pi\circ x\in U_\alpha$ and hence $x\in L\pi^{-1}(U_\alpha)$. On $\pi^{-1}(U_\alpha)$, let the connection 1-form $A=A_\alpha+d\theta_\alpha$. 
Then (using $K$ to denote $\dot x(s)$ for simplicity)
\be \int_{S^1}A(K)ds=\int_{S^1}d\theta_\alpha(K)ds+\int_{S^1}\widehat A_\alpha(K)ds=\int_{S^1}d\theta_\alpha(K)ds+\int_{S^1}\widehat A_\alpha(D\pi(K))ds, \ee
where $D\pi: TZ\to TX$ is the differential map of $\pi: Z\to X$.

Since $\pi\circ x$ lies in $U_\alpha$, we have the following composition of maps
\be S^1 \overset{x}{\to} \pi^{-1}(U_\alpha)\overset{\phi_\alpha}{\simeq} U_\alpha\times \TT\overset{p}{\to} \TT \ee and it is clear that
\be \int_{S^1}d\theta_\alpha(K)ds=\int_{S^1} (p\circ \phi_\alpha\circ x)^*(d\theta)\in \Z,\ee
where $\theta$ is the angular coordinate on $S^1$. 

Therefore $\int_{S^1}A(K)ds$ is an integer if and only if $\int_{S^1}\widehat A_\alpha(D\pi(K))ds$ is an integer. However we know that $\mathrm{hol}^Z(\pi\circ x)=e^{\int_{S^1}\widehat A_\alpha(D\pi(K))}ds$, so $\int_{S^1}A(K)ds$ is an integer if and only if $\mathrm{hol}^Z(\pi\circ x)$ is trivial. 
\end{proof}

Denote 
\be \label{q-loopspaceX} \LQ X:=\left\{\gamma: S^1\to X \left|\, \rmhol^Z(\gamma)=\mathrm{id}\right.\right\}.\ee
Given any $\gamma\in \LQ X$, solving the following equations
\be \label{liftingeqn}\left\{\begin{array}{cc}
                                     A(\dot{x}(s))=m\\
                                      D\pi(\dot{x}(s))=\dot\gamma(s)
                                     \end{array}
\right.
                                     \ee
with the initial value condition $\pi(x(0))=\gamma(0)$ gives a unique circle $x:[0, 1]\to Z$ that is in $\LQ_m Z$. Note that the triviality condition of $\rmhol^Z(\gamma)$ assures the existence of solution. For a fixed $\gamma$, various solutions differ by actions of elements in $\TT$ (shifting loops). It is also not hard to see that any loop in $\LQ_m Z$ that sits over $\gamma$ can be obtained in this way. Therefore, we see that $\LQ_m Z$ is actually a $\TT$-principal circle bundle over $\LQ X$:
\begin{equation}\label{bundlem}
\xymatrix
@=3pc
%@ur
{\TT \ar[r]&
\LQ_mZ  \ar[d]^{L\pi}   \\ 
&
\LQ X 
}
\end{equation}
\begin{remark} Generically $\LQ_mZ$ are infinitely dimensional and disconnected
since $\LQ X$ is generically infinitely dimensional and disconnected. 
\end{remark}

$\LQ X$ is covered by $\{(LU_\alpha)\cap \LQ X\}$. For each $\gamma\in (LU_\alpha)\cap \LQ X$, consider the circle values $e^{\int_0^1\theta_\alpha(\tilde \gamma(s))ds}$, where $\tilde\gamma(s)\in (L\pi)^{-1}(\gamma)$ and $A=A_\alpha+d\theta_\alpha$ on $\pi^{-1}(U_\alpha)$. If $\tilde\gamma_1=\lambda\cdot \tilde\gamma_2$, $\lambda\in \TT$, then 
$$e^{\int_0^1\theta_\alpha(\tilde \gamma_1(s))ds}=\lambda\cdot e^{\int_0^1\theta_\alpha(\tilde \gamma_2(s))ds}.$$
Denote 
\be \label{localbasis} s_{m, \alpha}(\gamma):=e^{\int_0^1\theta_\alpha(\tilde \gamma(s))ds}\ee
with $\theta_\alpha(\tilde \gamma(0))=0$. Varying $\gamma$, $s_{m, \alpha}$ gives a smooth circle valued function on $(LU_\alpha)\cap \LQ X$ and is a local section of the circle bundle (\ref{bundlem}). The transition functions of these local sections are 
\be \label{trans} f_{\alpha\beta}(\gamma)=\frac{s_{m, \alpha}(\gamma)}{s_{m, \beta}(\gamma)}=e^{\int_0^1\phi_{\alpha\beta}( \gamma(s))ds},\ee
where $\phi_{\alpha\beta}=\theta_\alpha-\theta_\beta$. Adopting the notations in Section \ref{reviewLM}, one can write $f_{\alpha\beta}=e^{\underline{\phi_{\alpha\beta}}}$.
\begin{proposition} $\underline A$, the average of the connection 1-form $A$ on $\LQ_m Z$, is a connection 1-form on the circle bundle (\ref{bundlem}). 
\end{proposition}
\begin{proof} We have
\be \label{conn} \underline {A_\alpha}-\underline {A_\beta}=-\underline{d\theta_\alpha}+\underline{d\theta_\beta}=-d\underline{\phi_{\alpha\beta}}=-d\ln f_{\alpha\beta}.\ee
The desired result follows. 
\end{proof}

%---------------------------------------------------------------------------------------------------------------------------------------------------------------------------------------------------------------------------------------------------

\subsection{The $(m, n)$-loop correspondence spaces and the vertical holonomy line bundles}

Now let us deal with the situation in the T-duality. Adopt the same notations as in Section 1. Let $(H, B_\alpha, F_{\alpha\beta}, (L_{\alpha\beta}, \nabla^{L_{\alpha\beta}}))$ be a gerbe with connection on $Z$, where $(H, B_\alpha, F_{\alpha\beta})$ denotes the Deligne class of the closed integral 3-form $H$ and $ (L_{\alpha\beta}, \nabla^{L_{\alpha\beta}})$ denotes the line bundles on double overlaps $\pi^{-1}(U_\alpha\cap U_\beta)$ that determines the gerbe. Different from Section \ref{ReviewT}, where one uses a good cover on $X$, here we use  a Brylinski cover of $X$ to deal with loop spaces. Nevertheless, since $H^i(U_{\alpha_I})=0$ for $i=2, 3$, an $H$-flux on $Z$ can still be geometrically realized as a gerbe with connection  $(H, B_\alpha, F_{\alpha\beta}, (L_{\alpha\beta}, \nabla^{L_{\alpha\beta}}))$ as in Section \ref{ReviewT}. The only difference is that here $H^2(U_{\alpha\beta}\times \TT)$ might not vanish and the line bundles $L_{\alpha\beta}$ is not necessarily trivial. Same as Section 1, we take $\nabla^{L_{\alpha\beta}}$ and $\nabla^{\widehat L_{\alpha\beta}}$ such that $\mathrm{hol}^{\rf}(\nabla^{L_{\alpha\beta}})=\widehat h_{\alpha\beta}$ and  $\mathrm{hol}^{\rf}(\nabla^{\widehat L_{\alpha\beta}})=h_{\alpha\beta}.$

\subsubsection{The $(m,n)$-loop correspondence spaces}\label{mnloop}
Denote
\be \label{dq-loopspaceX}\LDQ X:=\left\{\gamma: S^1\to X \left|\, \rmhol^Z(\gamma)=\mathrm{id}, \ \rmhol^{\widehat Z}(\gamma)=\mathrm{id}\right.\right\}.\ee To allow $Z$ side and the dual $\widehat Z$ side correspond, we need to restrict to the loop spaces 
\be  \label{dq-loopspaceZ}\LDQ Z:= \left\{x: S^1\to Z \left|\, A(\dot x(s))\in \Z, \forall s\in [0, 1]\ \mathrm{and}\ \pi\circ x\in \LDQ X\right. \right\}\ee
and 
 \be \label{dq-loopspacehatZ}\LDQ \widehat Z:= \left\{\widehat x: S^1\to \widehat Z \left|\, \widehat A(\dot{\widehat x}(s))\in \Z, \forall s\in [0, 1]\ \mathrm{and}\ \widehat \pi\circ \widehat x\in \LDQ X\right. \right\}.\ee
We also call $ \LDQ Z$ and $\LDQ\widehat Z$ {\bf winding quantized loop spaces} or {\bf wingding q-loop space} to abuse notation. 

Let $ \LDQ_m Z :=\left\{ x \in  \LDQ Z\left|\, A(\dot{ x}(s)) = m, \forall s\in [0, 1]\right\}\right.$ and similarly 
$\LDQ_n \widehat Z$.

$\forall m, n\in \Z$, we have the following picture
\begin{gather} 
\begin{aligned}
\xymatrix{
\LDQ_m Z \ar[dr]_{L\pi} &&
\LDQ_n \widehat Z \ar[dl]^{L\widehat \pi}\\
&\LDQ X
}
\end{aligned}
\end{gather}

Define the {\bf $(m, n)$-loop correspondence space} $\cL(m, n)$ by
\be \cL(m , n):=\LDQ_m Z\times_{\LDQ X} \LDQ_n \widehat Z.\ee 
Averaging on both sides of (\ref{HhatH}), we have 
\be \underline{\widehat H} =  \underline H + \underline{d(A\wedge \widehat A)}=\underline H + d\left(\underline{A\wedge \widehat A}\right). \ee
Then one has the following commutative diagram, the loop version of  (\ref{eqn:correspondence}),
\begin{equation} \label{eqn:loopcorrespondence}
\xymatrix @=6pc @ur { (\LDQ_m Z, [\underline H])\ar[d]_{L\pi} &
(\cL(m, n), [\underline H]=[\underline{\widehat H}])\ar[d]_{\widehat p_n} \ar[l]^{p_m} \\ \LDQ X & (\LDQ_n \widehat Z, [\underline{\widehat H}])\ar[l]^{L\widehat \pi}}
\end{equation}

\subsubsection{The vertical holonomy line bundles and a structure theorem} On the loop space $LZ$, motivated by  the holonomy line bundle $\cL_H$, we will consider the vertical holonomy line bundle defined using the fiber structure of $Z$ as follows. For any loop $x\in L\pi^{-1}(U_\alpha)\cap L\pi^{-1}(U_\beta)$, i.e. $x: S^1\to \pi^{-1}(U_\alpha)\cap \pi^{-1}(U_\beta)$, consider the vertically parallel transport equation of the line bundle with connection $(L_{\alpha\beta},\nabla^{L_{\alpha\beta}})$,
\be \label{hptran}  \nabla_{\dot{x}(s)^\rv} \eta^0_s=0,\ \ \ \ \ \ \eta^0_0=\mathrm{Id}, \ee
where $\eta^0_s\in \Hom(L_{\alpha\beta}|_{x(0)}, L_{\alpha\beta}|_{x(s)})$ and $\dot{x}(s)^\rv$ is the vertical projection  (with respect to the connection $A$ on $Z$) of tangent vector of the loop at time $s$. The holonomy of this vertically parallel transport gives a smooth function 
\be g^\rv_{\alpha\beta}=\mathrm{hol}^\rv(\nabla^{L_{\alpha\beta}}):=\Tr \eta^0_1\ee
on $L\pi^{-1}(U_\alpha)\cap L\pi^{-1}(U_\beta)$. Similar to the holonomy case, we have
\be g^\rv_{\alpha\beta}g^\rv_{\beta\gamma}g^\rv_{\gamma\alpha}=1.\ee
Consequently we obtain a complex line bundle  $\cL_H^\rv$ on $LZ$, which we call {\bf vertical holonomy line bundle}. It can be explicitly defined as
\be \cL_H^\rv= \left(\coprod_{\alpha\in I} \{\alpha\}\times L\pi^{-1}(U_\alpha)\times \CC\right)\bigg/\sim,  \ee
where
\be (\beta, x, w)\sim  (\alpha, x, g^\rv_{\alpha\beta}(x)\cdot w), \ \ \forall \alpha, \beta\in I,\,  x\in L(\pi^{-1}(U_\alpha\cap U_\beta)), w\in \CC. \ee
Denote by $s^\rv_\alpha=(x, 1)$ the local section of $\cL_H^\rv$ on $L\pi^{-1}(U_\alpha)$. 

Dually on the $\widehat Z$ side, we have the line bundle $\cL_{\widehat H}^\rv$ and local sections $\{\widehat s^\rv_\alpha\}$. 

Unlike the holonomy line bundle, on the vertical holonomy line bundle $\cL_{H}^\rv$, one does not have a good theory of connections and super flatness as on the holonomy line bundle $\cL_{H}$. Nevertheless, when restricting to the winding q-loop spaces, we have the following nice results. 

Denote by $\xi_m$ the circle bundle 
\begin{equation}
\xymatrix
@=3pc
%@ur
{\TT \ar[r]&
\LDQ_mZ  \ar[d]^{L\pi}   \\ 
&
\LDQ X 
}
\end{equation}
This is simply by restricting the circle bundle (\ref{bundlem}) to $\LDQ X$. By abusing notation, still denote by $\xi_m$ the complex line bundle over $\LDQ X$ associated to this circle bundle and the  standard representation of the circle on complex plane. Dually denote by $\widehat \xi_n$ the circle bundle
\begin{equation}
\xymatrix
@=3pc
%@ur
{\LDQ_n\widehat Z \ar[d]_{L\widehat \pi}  &
\widehat \TT \ar[l]  \\ 
 \LDQ X&
}
\end{equation}
and still denote by $\widehat \xi_n$ the complex line bundle over $\LDQ X$ associated to this circle bundle and the  standard representation of the circle on complex plane.

\begin{theorem} \label{struc}On $\LDQ_n \widehat{Z}$, one has an isomorphism of line bundles 
\be l_{m, n}: (L\widehat \pi)^*(\xi_m^{\otimes n})\to \cL_{\widehat H}^\rv;\ee 
on $\LDQ_m Z$, one has an isomorphism of line bundles 
\be \widehat l_{n, m}:  (L\pi)^*({\widehat\xi}_n^{\otimes m}) \to \cL_{H}^\rv.\ee
\end{theorem}
\begin{proof} For the line bundle $\cL_{\widehat H}^\rv$, the transition function $\widehat g^\rv_{\alpha\beta}$ on the intersection $$L\pi^{-1}(U_\alpha)\cap L\pi^{-1}(U_\beta)\cap \LDQ_n \widehat Z$$ is equal to $\mathrm{hol}^\rv(\nabla^{\widehat L_{\alpha\beta}})$. It is not hard to see that for any loop $\widehat x(s)\in \LDQ_n \widehat Z$, we have $\dot{\widehat x}(s)^\rv=n\widehat v|_{x(s)}, \forall s\in [0, 1]$.  Since $\mathrm{hol}^{\rf}(\nabla^{\widehat L_{\alpha\beta}})=h_{\alpha\beta}=e^{\phi_{\alpha\beta}},$ we therefore have 
\be \widehat g^\rv_{\alpha\beta}=\mathrm{hol}^{\rv}(\nabla^{\widehat L_{\alpha\beta}})=e^{n\,\underline{\phi_{\alpha\beta}}}=f_{\alpha\beta}^n,\ee
where $f_{\alpha\beta}$ is defined in (\ref{trans}).
Then the bundle map $l_{m, n}$ sending the local section $ (L\widehat \pi)^*(s_{m, \alpha}^{\otimes n})$ to $\widehat s^\rv_\alpha$ gives the desired isomorphism. 

The dual side can be similarly proved with a notation system of the hat-version of (\ref{localbasis}), (\ref{trans}) and (\ref{conn}). 

\end{proof}

\subsubsection{Super flatness and vanishing theorem} Let $K$ be the vector field of rotating loops on $LZ$. Let $K^\rv$ be the projection of $K$ on the vertical direction. Dually we have the vector field $\widehat K^\rv$ on $L\widehat Z$. Clearly, $\forall k\in \mathbb{Z}$, restricting on $\LDQ_k Z$, $K^\rv=k\bfv$ and on $\LDQ_k \widehat Z$, $\widehat K^\rv=k\widehat \bfv$. 
One has $\iota_{K^\rv} \underline {B_\alpha}=-k \underline {\widehat A_\alpha} $ since $\iota_v B_\alpha=-\widehat A_\alpha$. From the proof of Theorem \ref{struc}, we see that the transition function of the line bundle $\cL^\rv_H$ on $\LDQ_k Z$ is $g^\rv_{\alpha\beta}=e^{k\,\underline{\widehat \phi_{\alpha\beta}}}$, therefore 
\be (-\iota_{K^\rv} \underline {B_\alpha})-(-\iota_{K^\rv} \underline {B_\beta})=k\left(\underline {\widehat A_\alpha}-\underline {\widehat A_\beta}\right)=-d\ln g^\rv_{\alpha\beta}.\ee 
We therefore see that the 1-forms $\{-\iota_{K^\rv} \underline {B_\alpha}\}$ form a connection on $\cL^\rv_H$ over $\LDQ_k Z$ under the local basis $\{s^\rv_\alpha\}$.
Denote this connection by $\nabla^{\cL^\rv_H}$. Dually, the 1-forms $\{-\iota_{\widehat K^\rv} \underline {\widehat{B_\alpha}}\}$ form a connection on $\cL^\rv_{\widehat H}$ over $\LDQ_k \widehat{Z}$ under the local basis $\{\widehat s^\rv_\alpha\}$. Denote this connection by $\nabla^{\cL^\rv_{\widehat H}}$. 

From the proof of Theorem \ref{struc}, it is not hard to see that 
\be l_{m,n}^*(\nabla^{\cL^\rv_H} )=(L\widehat\pi)^*((\nabla^{\xi_m})^{\otimes n})), \ \ \widehat l_{n, m}^*(\nabla^{\cL^\rv_{\widehat H}})=(L\pi)^*( (\nabla^{\widehat \xi_n})^{\otimes m}),\ee
where $\nabla^{\xi_m}$ is the connection on the line bundle $\xi_m$ determined by $\underline A$ and the the  standard representation of the circle on complex plane, and $\nabla^{\widehat \xi_n}$ has the similar meaning.

Consider the line bundle $\cL^\rv_H$ on $\LDQ_k Z$. Since the transition functions of this line bundle are $\{e^{k\,\underline{\widehat{\phi}_{\alpha\beta}}}\}$, which are purely functions on $\LDQ X$, one can define the Lie derivative $L_{K^\rv}$ on $\Omega^*(\LDQ_k Z, \cL^\rv_{H})$ by acting only on the form part. Similarly one can define the Lie derivative $L_{\widehat K^\rv}$ on $\Omega^*(\LDQ_k\widehat Z,\cL^\rv_{\widehat H})$.

\begin{proposition}  \label{flat} $\forall k\in \mathbb{Z}$, we have: \newline
on $\Omega^*(\LDQ_k Z, \cL^\rv_{H})$, the following identity holds,
\be (\nabla^{\cL^\rv_H}-\iota_{K^\rv}+\underline H)^2+L_{K^\rv}=0; \ee
on $\Omega^*(\LDQ_k\widehat Z,\cL^\rv_{\widehat H})$, the following identity holds,
\be (\nabla^{\cL^\rv_{\widehat H}}-\iota_{\widehat K^\rv}+\underline {\widehat H})^2+L_{\widehat K^\rv}=0. \ee
\end{proposition}
\begin{proof} Choose local basis $\{s^\rv_\alpha\}$. To prove the first identity, it suffices to prove that 
\be (d-\iota_{K^\rv} \underline {B_\alpha}-\iota_{K^\rv} +\underline H)^2+L_{K^\rv}=0.\ee
This holds as we have
$$(d-\iota_{K^\rv})^2+L_{K^\rv}=0, $$
$$d\underline{H}=0, \ \ \iota_{K^\rv}\iota_{K^\rv} \underline {B_\alpha}=0, $$
$$ -d\iota_{K^\rv} \underline {B_\alpha}-\iota_{K^\rv}\underline H=k\, d\underline {\widehat A_\alpha} -k\underline{\iota_{v}H}=k\,( \underline{d\widehat A_\alpha}-\underline{\iota_{v}H})=0.$$

One can similarly prove the second identity. 

\end{proof}

Denote
\be \gA^{\rv, \bullet}(\LDQ_kZ)=\Omega^\bullet(\LDQ_kZ,\cL^\rv_H)^\TT[[u, u^{-1}]],\ee
\be \gA^{\rv, \bullet}(\LDQ_k\widehat Z)=\Omega^\bullet(\LDQ_k\widehat Z,\cL^\rv_{\widehat H})^{\widehat \TT}[[u, u^{-1}]]. \ee
Denote the cohomologies of the following complexes 
\be (\gA^{\rv, \bullet}(\LDQ_kZ), \nabla^{\cL^\rv_H}-u\iota_{K^\rv}+u^{-1}\underline H),\ee
\be (\gA^{\rv, \bullet}(\LDQ_k\widehat Z), \nabla^{\cL^\rv_{\widehat H}}-u\iota_{\widehat K^\rv}+u^{-1}\underline {\widehat H}) \ee
by  
\be h^{\rv, \bu}(\LDQ_k Z, \nabla^{\cL_H^\rv}:\underline H),\ee
\be h^{\rv, \bu}(\LDQ_k \widehat Z, \nabla^{\cL_{\widehat H}^\rv}: \underline {\widehat H})\ee
respectively. 

When $k=0$, $K^\rv$ and $\widehat K^\rv$ are just 0, and from Theorem \ref{struc} as well as its proof, the line bundles $\cL_H^\rv$ and $\cL_{\widehat H}^\rv$ are trivial with trivial connections, and hence the complexes are simply
\be (\Omega^\bullet(\LDQ_0Z)^\TT[[u, u^{-1}]], d+u^{-1}\underline{H}), \ee
\be (\Omega^\bullet(\LDQ_0\widehat Z)^{\widehat \TT}[[u, u^{-1}]], d+u^{-1}\underline{\widehat H}). \ee
 
When $k\neq 0$, we have the following acyclic result, 
\begin{theorem}\label{acyc} If $k\neq 0$, then 
\be h^{\rv, \bu}(\LDQ_k Z, \nabla^{\cL_H^\rv}:\underline H)\cong \{0\},\ee
\be h^{\rv, \bu}(\LDQ_k \widehat Z, \nabla^{\cL_{\widehat H}^\rv}: \underline {\widehat H}) \cong \{0\}.\ee
\end{theorem}
\begin{proof} As shown in Section \ref{loopprincipal}, $\LDQ_k Z$ is the total space of the circle bundle
\begin{equation}
\xymatrix
@=3pc
%@ur
{\TT \ar[r]&
\LDQ_kZ  \ar[d]^{L\pi}   \\ 
&
\LDQ X 
}
\end{equation}
and $\underline A$ is a connection. The curvature of the connection $\underline A$ is $\underline{dA}=\underline{F_A}$.  

For $k\neq 0$, set $\eta_k=\frac{u^{-1}\underline{A}}{u^{-1}\underline{F_A}-k}$. We have
\be 
\begin{split}
&(d-u\iota_{K^\rv})\eta_k\\
=&\frac{[(d-u\iota_{K^\rv})(u^{-1}\underline{A})](u^{-1}{\underline{F_A}-k)-u^{-1}\underline{A}\, [(d-u\iota_{K^\rv})(u^{-1}\underline{F_A}-k)]}}{(u^{-1}\underline{F_A}-k)^2}\\
=&\frac{(u^{-1}\underline{F_A}-k)^2}{(u^{-1}\underline{F_A}-k)^2}\\
=&1.
\end{split}
\ee

We therefore obtain a chain homotopy for $k\neq 0$:  $\forall x\in \Omega^\bullet(\LDQ_kZ,\cL^\rv_H)[[u, u^{-1}]],$
\be 
\begin{split}
&(\nabla^{\cL_H^\rv}-u\iota_{K^\rv}+u^{-1}\underline H)(\eta_nx)+\eta_n[(\nabla^{\cL_H^\rv}-u\iota_{K^\rv}+u^{-1}\underline H)x]\\
=&[(d-u\iota_{K^\rv})\eta_n]x\\
=&x.
\end{split}
\ee
We can prove the second isomorphism in a verbatim way by using $\underline{\widehat{A}}$ and $\underline{F_{\widehat{A}}}$ to construct the homotopy. 
\end{proof}

%%%%%%%%%%%%%%%%%%%%%%%%%%%%%%%%%%%%%%%%%%%
\section{Loop Hori formulae} \label{secloopHori}

%----------------------------------------------------------------------------------------------------------------------------------------------------

\subsection{Twisted integration along the fiber} In order to construct the loop Hori formulae in the next subsection, we first briefly review the {\bf twisted integration along the fiber} introduced in \cite{HM18}.

Let $\pi: P\to M$ be a principal circle bundle over $M$, and $\Theta$ a connection one form on $P$. Let $L$ be a Hermitian line bundle over $M$ such that $Z$ is the circle bundle of $L$. Let $\nabla^L$ be the connection on $L$ corresponding to $\Theta$. Choose a good cover $\{U_\alpha\}$ on $M$ such that $\pi^{-1}(U_\alpha)\cong U_\alpha\times S^1.$ Let $\{f_\alpha\}$ be a local basis of $L$ corresponding to the constant map $U_\alpha\to \{1\}\subset S^1.$ 

$\forall n\in \ZZ$, the twisted integration along the fiber is defined as follows: for $\omega\in \Omega^*(P)$, 
\be  \int^{P/M, n}\!\!\!\!\!\omega\in \Omega^*(M, L^{\otimes n}), \ \mathrm{such\ that}\   \left.\left(\int^{P/M, n}\!\!\!\!\! \omega\right)\right|_{U_\alpha}=\left(\int^{\pi^{-1}(U_\alpha)/U_\alpha}\omega_\alpha e^{2\pi in\lambda_\alpha}\right)\otimes f_\alpha^{\otimes n},   \ee
where $\omega_\alpha=\omega|_{\pi^{-1}(U_\alpha)}$, $\lambda_\alpha$ is the vertical coordinates of $\pi^{-1}(U_\alpha).$ Note that as 
on $U_{\alpha\beta}=U_\alpha\cap U_\beta$,  $f_\alpha/f_\beta=e^{2\pi i(\lambda_\beta-\lambda_\alpha)}$ (a function on $U_{\alpha\beta}$), the above construction patches to be a global section of the bundle $\wedge^*(T^*M)\otimes L^{\otimes n}$. Moreover, it is not hard to see that this definition is independent of choice of the good cover $\{U_\alpha\}$ and local trivializations.

\begin{theorem}[\protect \cite{HM18}] \label{interchange} Let $Y$ be a vector field on $M$ and \, $\widetilde{Y}$ a lift of $Y$ on $P$. Let $H$ be a differential form on $M$. Then $\forall n\in \ZZ$
\be(\nabla^{L\otimes n}-\iota_Y+H)\int^{P/M, n}\!\!\!\!\!\omega=-\int^{P/M, n}\!\!\!\!\! (d+n\Theta-\iota_{\widetilde{Y}}+H)\omega.  \ee

\end{theorem}

%---------------------------------------------------------------------------------------------------------------------------------------------------------------------------------------------------------------------------------------------------

\subsection{Loop Hori formulae}

Denote
\be \gB_{-j}^{\bullet}(\LDQ_k Z):=\{\omega\in \Omega^{\bullet}(\LDQ_kZ)[[u, u^{-1}]]|\, L_{\bfv}\omega=-j\omega \},\ee
\be \gB_{-j}^{\bullet}(\LDQ_k \widehat Z):=\{\widehat \omega\in \Omega^{\bullet}(\LDQ_k\widehat Z)[[u, u^{-1}]]|\, L_{\widehat \bfv}\widehat \omega=-j\widehat \omega \}.\ee

Let $\omega\in \gB_{-n}^{\bullet}(\LDQ_m Z)$. Define the {\bf $(m, n)$-loop Hori formula} by 
\be\label{loopHori} \tau_{m,n}(\omega)= \left(l_{m, n}^{-1}\right)^*\left(\int^{\TT, n}\omega\, e^{-u^{-1}\underline {A\wedge\widehat{A}}}\right)\in \Omega^{\bullet+1}(\LDQ_n \widehat Z,  \cL_{\widehat H}^\rv)[[u, u^{-1}]], \ee
where $\int^{\TT, n}$ stands for $\int^{\cL(m, n)/\LDQ \widehat Z, n}$ for simplicity.

Denote by $\rho_m$ the tautological global section of the line bundle $(L\widehat\pi\circ \widehat p_n)^*\xi_m$ on $\cL(m, n)$. Let $\widehat \theta\in \gA^{\rv, \bullet}(\LDQ_n\widehat Z)=\Omega^\bullet(\LDQ_n\widehat Z,\cL^\rv_{\widehat H})^{\widehat \TT}[[u, u^{-1}]]$. Define the {\bf inverse $(m, n)$-loop Hori formula} by
\be \label{ivloopHori}\widehat{\sigma}_{n, m}(\widehat \theta)= \int^{\widehat\TT}\widehat p^*_n(l_{m, n}^*(\widehat \theta)) \cdot (\rho_m^{-1})^{\otimes n} \cdot e^{u^{-1}\underline{A\wedge\widehat{A}}}\in \Omega^{\bullet+1}(\LDQ_mZ)[[u, u^{-1}]]. \ee

One can dually define the {\bf $(n, m)$-loop Hori formula} $\widehat\tau_{n, m}$ from $\widehat Z$ to $Z$ and the {\bf inverse $(n, m)$-loop Hori formula} $\sigma_{m, n}$ from $Z$ to $\widehat Z$.

\begin{theorem}\label{main} (i) $\tau_{m,n}(\omega)\in  \gA^{\rv, \bullet+1}(\LDQ_n\widehat Z)$ and 
\be \tau_{m,n}: \gB_{-n}^{\bullet}(\LDQ_m Z)\to  \gA^{\rv, \bullet+1}(\LDQ_n\widehat Z)\ee
is an isomorphism; $\widehat{\sigma}_{n, m}(\widehat \theta)\in \gB_{-n}^{\bullet+1}(\LDQ_m Z)$ and 
\be\widehat{\sigma}_{n, m}: \gA^{\rv, \bullet}(\LDQ_n\widehat Z)\to \gB_{-n}^{\bullet+1}(\LDQ_m Z) \ee
is an  isomorphism; $\widehat{\sigma}_{n, m}\circ\tau_{m,n}=-\mathrm{Id}, \  \tau_{m,n}\circ \widehat{\sigma}_{n, m}=-\mathrm{Id}.$ The dual results for $\widehat\tau_{n, m}$ and  $\sigma_{m, n}$ are also true. \newline
(ii) The map $ \tau_{m,n}$ induces a chain map on the complexes 
$$(\gB_{-n}^{\bullet}(\LDQ_m Z), d+u^{-1}\underline{H})\to ( \gA^{\rv, \bullet+1}(\LDQ_n\widehat Z),-( \nabla^{\cL^\rv_{\widehat H}}-u\iota_{\widehat K^\rv}+u^{-1}\underline {\widehat H}))$$ and the map $\widehat \sigma_{n, m}$ induces induces a chain map on the complexes 
$$(\gA^{\rv, \bullet}(\LDQ_n\widehat Z), \nabla^{\cL^\rv_{\widehat H}}-u\iota_{\widehat K^\rv}+u^{-1}\underline {\widehat H})\to(\gB_{-n}^{\bullet+1}(\LDQ_m Z), -(d+u^{-1}\underline{H})).$$ The dual results for $\widehat\tau_{n, m}$ and  $\sigma_{m, n}$ are also true. 
\end{theorem}
\begin{proof} (i) Let $\{s_{m,\alpha}\}$ be local sections of of the line bundle $\xi_m$ and $\{\widehat{s}_{n,\alpha}\}$ local sections of of the line bundle $\widehat{\xi_n}$ as described in Section \ref{loopprincipal}. Let $\lambda_\alpha$ be the vertical coordinate function of $(L\pi)^{-1}(LU_\alpha\cap \LDQ X)$ and $\widehat{\lambda}_\alpha$ the vertical coordinate function of $(L\widehat{\pi})^{-1}(LU_\alpha\cap \LDQ X)$. 

Take $\omega\in \gB_{-n}^{\bullet}(\LDQ_m Z)$. Locally, $\omega$ is of the form
$$(\omega_{\alpha, 0}+ \omega_{\alpha, 1}\underline{A})e^{-2\pi in\lambda_\alpha},$$ where $\omega_{\alpha, 0}$ and  $\omega_{\alpha, 1}$ are both forms on $LU_\alpha\cap \LDQ X$. Then 
\be \label{tau}
\begin{split}
&\left.\tau_{m,n}(\omega)\,\right|_{(L\widehat\pi)^{-1}(LU_\alpha\cap \LDQ X)}\\
=&\left. \left(l_{m, n}^{-1}\right)^*\left(\int^{\TT, n}\omega\, e^{-u^{-1}\underline {A\wedge\widehat{A}}}\right)\, \right|_{(L\widehat\pi)^{-1}(LU_\alpha\cap \LDQ X)}\\
=&\left(l_{m, n}^{-1}\right)^*\left\{\left( \int^{\TT}(\omega_{\alpha, 0}+ \omega_{\alpha, 1}\underline{A})e^{-u^{-1}\underline{A\wedge\widehat{A}}}e^{-2\pi in\lambda_\alpha}\cdot e^{2\pi i n\lambda_\alpha}\right)\otimes(L\widehat\pi)^*(s_{m,\alpha}^{\otimes n})\right\}\\
=&\left( \int^{\TT}(\omega_{\alpha, 0}+ \omega_{\alpha, 1}\underline A)e^{-u^{-1}\underline{A\wedge\widehat{A}}}\right)\otimes\widehat s^\rv_\alpha.
\end{split}
\ee

On $\cL(m, n)$, consider the form 
\be \kappa\!\left(\underline A, \underline {\widehat A}\right):=e^{u^{-1}(\underline A\wedge \underline{\widehat A}-\underline{A\wedge\widehat{A}})}. \ee
One can see that 
\be\iota_\bfv\kappa\!\left(\underline A, \underline {\widehat A}\right)=\kappa(\underline A, \underline {\widehat A})u^{-1} \left((\iota_\bfv{\underline A})\wedge \underline{\widehat A}-\underline{(\iota_vA)\wedge\widehat{A}}\right)=0,\ee
and 
\be\iota_{\widehat\bfv}\kappa\!\left(\underline A, \underline {\widehat A}\right)=-\kappa(\underline A, \underline {\widehat A})u^{-1} \left({\underline A}\wedge \iota_{\widehat\bfv}{\underline{\widehat A}}-\underline{A\wedge\iota_{\widehat v}\widehat{A}}\right)=0.\ee
Therefore $\kappa\!\left(\underline A, \underline {\widehat A}\right)$ is actually a form on $\LDQ X$. 

Hence 
\be 
\begin{split} \label{taumn}
&\left( \int^{\TT}(\omega_{\alpha, 0}+ \omega_{\alpha, 1}\underline A)e^{-u^{-1}\underline{A\wedge\widehat{A}}}\right)\otimes(\widehat s^\rv_\alpha)^{\otimes n}\\
=&\left( \int^{\TT}\left(\kappak\omega_{\alpha, 0}+ \kappak\omega_{\alpha, 1}\underline A\right)e^{-u^{-1}\underline{A}\wedge\underline{\widehat{A}}}\right)\otimes(\widehat s^\rv_\alpha)^{\otimes n},\\
\end{split}
\ee
which is equal to
$$-\left(\kappak\omega_{\alpha, 0}\,\underline{\widehat A}+\kappak\omega_{\alpha, 1}\right)\otimes(\widehat s^\rv_\alpha)^{\otimes n}$$
if $\omega$ is of even degree; or
$$\left(\kappak\omega_{\alpha, 0}\,\underline{\widehat A}+\kappak\omega_{\alpha, 1}\right)\otimes(\widehat s^\rv_\alpha)^{\otimes n}$$
if $\omega$ is of odd degree. So we see that
$$ \tau_{m,n}(\omega) \in \gA^{\rv, \bullet+1}(\LDQ_n\widehat Z). $$

On the other hand, if $\widehat \theta\in \gA^{\rv, \bullet}(\LDQ_n\widehat Z)$, suppose $\widehat\theta$ is locally equal to
$$\left(\widehat\eta_{\alpha, 0}\underline{\widehat A}+\widehat\eta_{\alpha, 1}\right)\otimes\widehat s^\rv_\alpha, $$
where $\widehat\eta_{\alpha, 0}, \widehat\eta_{\alpha, 1}$ are forms on $LU_\alpha\cap \LDQ X$.
Then $\widehat{\sigma}_{n, m}(\widehat\theta)$ is locally equal to 
\be \label{T hat}
\begin{split}
&\left. \int^{\widehat\TT}\widehat p^*_n(l_{m, n}^*(\widehat \theta)) \cdot (\rho_m^{-1})^{\otimes n} \cdot e^{u^{-1}\underline{A\wedge\widehat{A}}}\, \right|_{(L\pi)^{-1}(LU_\alpha\cap \LDQ X)}\\
=&\left(\int^{\widehat\TT}\left(\widehat\eta_{\alpha, 0}\underline{\widehat A}+\widehat\eta_{\alpha, 1}\right) \cdot e^{u^{-1}\underline{A\wedge\widehat{A}}}\right)\cdot e^{-2\pi in \lambda_\alpha}\\
=& \left(\int^{\widehat\TT}\left(\kappak^{-1}\widehat\eta_{\alpha, 0}\underline{\widehat A}+\kappak^{-1}\widehat\eta_{\alpha, 1}\right) \cdot e^{u^{-1}\underline{A}\wedge\underline{\widehat{A}}}\right)\cdot e^{-2\pi in \lambda_\alpha}
\end{split}
\ee which is equal to
$$\left(\kappak^{-1}\widehat\eta_{\alpha, 0}+\kappak^{-1}\widehat\eta_{\alpha, 1}\underline{A}\right)\cdot e^{-2\pi in \lambda_\alpha} $$if $\widehat\theta$ is of even degree; or
$$ -\left(\kappak^{-1}\widehat\eta_{\alpha, 0}+\kappak^{-1}\widehat\eta_{\alpha, 1}\underline{A}\right)\cdot e^{-2\pi in \lambda_\alpha} $$if $\widehat\theta$ is of odd degree.
And so evidently $L_\bfv \widehat{\sigma}_{n, m}(\widehat\theta)=-n\widehat{\sigma}_{n, m}(\widehat\theta)$. We therefore have
$$ \widehat{\sigma}_{n, m}(\widehat\theta)\in  \gB_{-n}^{\bullet}(\LDQ_m Z). $$

From the above expressions (\ref{taumn}), (\ref{T hat}) and local nature of (i), we see that $\tau_{m,n}, \widehat \sigma_{n,m}$ are both isomorphisms and 
\be \widehat{\sigma}_{n, m}\circ\tau_{m,n}=-\mathrm{Id}, \  \tau_{m,n}\circ \widehat{\sigma}_{n, m}=-\mathrm{Id}.\ee

The dual results for $\widehat\tau_{n, m}$ and  $\sigma_{m, n}$ can be proved similarly.

$\, $

(ii) We have 
\be [(d+u^{-1}\underline{H})\omega]e^{-u^{-1}\underline{A\wedge\widehat{A}}}=[d+u^{-1}\underline H-u^{-1}(\underline H-\underline {\widehat H})](\omega e^{-u^{-1}\underline{A\wedge\widehat{A}}})=(d+u^{-1}\underline{\widehat H})(\omega e^{-u^{-1}\underline{A\wedge\widehat{A}}}). \ee
Also one has
\be 
\begin{split}
&(n\underline A-u\iota_{\widehat K^\rv})(\omega e^{-u^{-1}\underline{A\wedge\widehat{A}}})\\
=&n\underline{A}\,\omega e^{-u^{-1}\underline{A\wedge\widehat{A}}}-u\iota_{\widehat K^\rv}(\omega e^{-u^{-1}\underline{A\wedge\widehat{A}}})\\
=&(-1)^{|\omega|}\omega(n\underline A-n\underline A)e^{-u^{-1}\underline{A\wedge\widehat{A}}}\\
=&0.  
\end{split}
\ee

By Theorem \ref{interchange}, we have
\be
\begin{split}
&\tau_{m,n}((d+u^{-1}\underline H)\omega)\\
=& \left(l_{m, n}^{-1}\right)^*\left(\int^{\TT, n}(d+u^{-1}\underline H)\omega\, e^{-u^{-1}\underline {A\wedge\widehat{A}}}\right)\\
=& \left(l_{m, n}^{-1}\right)^*\left(\int^{\TT, n}\left(d+u^{-1}\underline {\widehat H}\right)\left(\omega\, e^{-u^{-1}\underline {A\wedge\widehat{A}}}\right)\right)\\
=& \left(l_{m, n}^{-1}\right)^*\left(\int^{\TT, n}\left(d+n\underline A-u\iota_{\widehat K^\rv}+u^{-1}\underline {\widehat H}\right)\left(\omega\, e^{-u^{-1}\underline {A\wedge\widehat{A}}}\right)\right)\\
=&\left(l_{m, n}^{-1}\right)^*\left(-\left((L\widehat{\pi})^*\nabla^{\xi_m^{\otimes n}}-u\iota_{\widehat K^\rv}+u^{-1}\underline{\widehat{H}}\right)\int^{\TT, n} \omega\, e^{-u^{-1}\underline {A\wedge\widehat{A}}}\right) \\
=&-\left(\nabla^{\cL^\rv_{\widehat H}}-u\iota_{\widehat K^\rv}+u^{-1}\underline{\widehat{H}}\right)\tau_{m,n}(\omega).
\end{split}
\ee
Therefore $\tau_{m,n}$ is a chain map. 

As $-\widehat \sigma_{n, m}$ is the inverse of $\tau_{m,n}$ , one deduces that $\widehat \sigma_{n,m}$ is also a chain map. 

The dual results for $\widehat\tau_{n, m}$ and  $\sigma_{m, n}$ can be proved similarly.
\end{proof}

As an immediate consequence, we have
\begin{corollary}  $\forall m, n\in \ZZ$, 
\be H(\Omega^{\bullet}(\LDQ_mZ)^\TT[[u, u^{-1}]], d+u^{-1}\underline{H})\cong H(\Omega^{\bullet+1}(\LDQ_n\widehat Z)^{\hat \TT}[[u, u^{-1}]], -(d+u^{-1}\underline{\widehat H})).\ee
\end{corollary}
\begin{proof} From Theorem \ref{main}, $\forall m\in \ZZ$, $\tau_{m,0}$ induces an isomorphism,
\be H(\Omega^{\bullet}(\LDQ_mZ)^\TT[[u, u^{-1}]], d+u^{-1}\underline{H})\cong H(\Omega^{\bullet+1}(\LDQ_0\widehat Z)^{\hat \TT}[[u, u^{-1}]], -(d+u^{-1}\underline{\widehat H})),\ee
and $\forall n\in \ZZ$, $\widehat \tau_{n, 0}$ induces an isomorphism 
\be H(\Omega^{\bullet+1}(\LDQ_n\widehat Z)^{\hat \TT}[[u, u^{-1}]], -(d+u^{-1}\underline{\widehat H}))\cong H(\Omega^{\bullet}(\LDQ_0Z)^\TT[[u, u^{-1}]], d+u^{-1}\underline{H}).\ee
The desired results follows. 
\end{proof}

Combining Theorem \ref{acyc} and Theorem \ref{main}, we have
\begin{corollary} If $n\neq 0$, 
\be H(\gB_{-n}^{\bullet}(\LDQ_m Z), d+u^{-1}\underline{H})\cong \{0\}, \forall m\in \ZZ;\ee
if $m\neq 0$, 
\be H(\gB_{-m}^{\bullet}(\LDQ_n \widehat Z), d+u^{-1}\widehat{\underline{H}})\cong \{0\}, \forall n\in \ZZ.\ee
\end{corollary}

To summarize, we have the following isomorphisms as chain maps,
\[ 
\begin{tikzcd}
\gB_{-n}^{\bullet}(\LDQ_m Z)\oplus\gA^{\rv, \bullet}(\LDQ_mZ)\ar[rr, shift left]{rr}{\tau_{m, n}\oplus \sigma_{m, n}} & & \gA^{\rv, \bullet+1}(\LDQ_n\widehat Z)\oplus\gB_{-m}^{\bullet+1}(\LDQ_n \widehat Z)
\ar[ll, shift left]{ll}{\widehat \sigma_{n, m}\oplus \widehat \tau_{n, m}}\end{tikzcd}
\] 
such that 
$$\widehat{\sigma}_{n, m}\circ\tau_{m,n}=-\mathrm{Id}, \  \tau_{m,n}\circ \widehat{\sigma}_{n, m}=-\mathrm{Id} $$
and 
$$\sigma_{m, n}\circ\widehat\tau_{n,m}=-\mathrm{Id}, \  \widehat\tau_{n,m}\circ \sigma_{m, n}=-\mathrm{Id}. $$

%---------------------------------------------------------------------------------------------------------------------------------------------------------------------------------------------------------------------------------------------------

\subsection{Relation to the exotic Hori formulae} In this subsection, we discuss the relation between the exotic Hori formula introduced in \cite{HM18} and the loop Hori formula. We first briefly review the exotic Hori formulae in this paper. 

Let $\xi, \hat{\xi}$ be the complex line bundle determined by the circle bundles $Z, \hat{Z}$ and the standard representation of the circle on the complex plane. Let $\nabla^\xi$ and $\nabla^{\hat{\xi}}$ be the connections on $\xi, \hat{\xi}$ induced from the connections on $Z, \hat{Z}$ respectively.

\begin{theorem}[\protect \cite{HM18}]\label{main-exotic1} $\forall n\in \mathbb{Z}$, we have: \newline
on $\Omega^*(Z, \pi^*(\hat{\xi}^{\otimes n}))$, the following identity holds,
\be (\pi^*\nabla^{\hat{\xi}^{\otimes n}}-\iota_{nv}+H)^2+nL_v=0; \ee
on $\Omega^*(\hat{Z}, \pi^*(\xi^{\otimes n}))$, the following identity holds,
\be (\hat{\pi}^*\nabla^{\xi^{\otimes n}}-\iota_{n\hat{v}}+\hat{H})^2+nL_{\hat{v}}=0. \ee
\end{theorem}

Denote
$$\mathcal{A}^*(Z)=\bigoplus_{n\in \ZZ} \mathcal{A}_n^*(Z):=\bigoplus_{n\in \ZZ}\Omega^*(Z, \pi^*(\hat{\xi}^{\otimes n}))^{\TT}[[u, u^{-1}]]. $$ 
$$\mathcal{A}^*(\hat Z)=\bigoplus_{n\in \ZZ} \mathcal{A}_n^*(\hat Z):=\bigoplus_{n\in \ZZ}\Omega^*(\hat Z, \hat \pi^*(\xi)^{\otimes n}))^{\hat \TT}[[u, u^{-1}]]. $$
For each $n\in \ZZ$, consider the complexes
$$(\mathcal{A}_n^*(Z), \pi^*\nabla^{\hat{\xi}^{\otimes n}}-u\iota_{nv}+u^{-1}H),$$
$$ (\mathcal{A}_n^*(\hat Z), \hat{\pi}^*\nabla^{\xi^{\otimes n}}-u\iota_{n\hat{v}}+u^{-1}\hat{H}).$$

Denote
\be \Omega^{\bullet}_{-n}(Z)[[u, u^{-1}]]=\{\omega \in \Omega^{\bullet}(Z)[[u, u^{-1}]]| L_{v}\omega=-n\omega\}.\ee
Note that $\Omega^{\bullet}_{-n}(Z)[[u, u^{-1}]]=\Omega^{\bullet}(Z)^\TT[[u, u^{-1}]]$, i.e. the $\TT$-invariant forms on $Z$. 

Let $\mu\in \Omega^{\bullet}_{-n}(Z)[[u, u^{-1}]].$ Define the {\bf exotic Hori formula} by
\be\label{exoticHori} \tau_n(\mu)= \int^{\TT, n}\mu\, e^{-u^{-1}A\wedge\hat{A}}\in \Omega^{\bullet+1}(\hat Z, \hat\pi^*(\xi)^{\otimes n}))[[u, u^{-1}]],\ee
where $\int^{\TT, n}$ stands for $\int^{(Z\times_X \hat{Z})/\hat{Z}, n}$ for simplicity.

Denote by $\rho$ the tautological global section of the line bundle $(\hat\pi\circ \hat p)^*\xi$ on $Z\times_X \hat{Z}$. Let $\hat \nu\in \Omega^{\bullet}(\hat Z, \hat\pi^*(\xi)^{\otimes n})^{\hat \TT}[[u, u^{-1}]]$. Define the {\bf inverse exotic Hori formula} by
\be \label{ivexoticHori}\hat{\sigma}_n(\hat \nu)= \int^{\hat\TT}\hat p^*(\hat \theta) \cdot (\rho^{-1})^{\otimes n} \cdot e^{u^{-1}A\wedge\hat{A}}\in \Omega^{\bullet+1}(Z)[[u, u^{-1}]]. \ee

One can dually define the $\widehat \tau_n$ from $\widehat Z$ to $Z$ and the $\sigma_n$ from $Z$ to $\widehat Z$.

\begin{theorem}[\protect \cite{HM18}] \label{main-exotic2}
(i) $\tau_n(\mu)\in  \Omega^{\bullet+1}(\hat{Z}, \hat{\pi}^*(\xi^{\otimes n}))^{\hat \TT}[[u, u^{-1}]]$ and $\hat{\sigma}_n(\nu)\in \Omega^{\bullet+1}_{-n}(Z)[[u, u^{-1}]].$ $\tau_n$ induces an isomorphism  
$$\tau_n: \Omega^*_{-n}(Z)[[u, u^{-1}]]\to \Omega^*(\hat{Z}, \hat{\pi}^*(\xi^{\otimes n}))^{\hat \TT}[[u, u^{-1}]]$$
and $\hat \sigma_n=-\tau_n^{-1}$.The dual results for $\hat\tau_n$ and $\sigma_n$ are also true.\newline
(ii) The map $\tau_n$ induces a chain map on the complexes 
$$(\Omega^{\bullet}_{-n}(Z)[[u, u^{-1}]], d+u^{-1}H)\to (\Omega^{\bullet+1}(\hat{Z}, \hat{\pi}^*(\xi^{\otimes n}))^{\hat \TT}[[u, u^{-1}]],-(\hat{\pi}^*\nabla^{\xi^{\otimes n}}-u\iota_{n\hat{v}}+u^{-1}\hat{H}))$$ and the map $\hat \sigma_n$ induces induces a chain map on the complexes 
$$(\Omega^{\bullet}(\hat{Z})[[u, u^{-1}]], \hat{\pi}^*(\xi^{\otimes n})),\hat{\pi}^*\nabla^{\xi^{\otimes n}}-u\iota_{n\hat{v}}+u^{-1}\hat{H})^{\hat \TT}\to(\Omega^{\bullet+1}_{-n}(Z), -(d+u^{-1}H)).$$ The dual results for $\hat\tau_n$ and $\sigma_n$ are also true. 
\end{theorem}

To compare the loop Hori formulae and the exotic Hori fomulae, we have the following commutative diagram,
\be
\xymatrix@=3.5pc{
 \cL(0, n)&& \LDQ_n\widehat Z\\
    &\LDQ_0 Z&& \LDQ X\\
    Z\times_X \widehat Z&&\widehat Z\\
    &Z&&X\\
    \ar"1,1";"2,2"_{p_0}
    \ar"1,3";"2,4"_{L\widehat\pi}
    \ar"3,1";"4,2"_{p}
    \ar"3,3";"4,4"^{\widehat \pi}
    \ar"1,1";"1,3"^{\widehat p_n}
    \ar"2,2";"2,4"_{L\pi\ \ \ \ \ \ \ }
    \ar"4,2";"4,4"_{\pi}
    \ar"3,1";"1,1"
    \ar"4,2";"2,2"
    \ar"4,4";"2,4"
   \ar"3,1";"3,3"^{\ \ \ \ \ \ \ \ \ \ \ \  \widehat p}|\hole\ \ \ \ \ \ \ \ \ %直接从(3,1)指向(3,3),带空洞
    \ar"3,3";"1,3"|\hole
}
\ee
where the vertical maps $i_{0,Z}: Z\to \LDQ_0Z$ and $i_X: X\to \LDQ X$ are embeddings of the constant loops in the loop spaces, the map on the hat side is
$$\iota_{n, \widehat Z}: \widehat Z\to \LDQ_n\widehat Z, \ \ \ x\to \gamma_x\ \ \mathrm{such\ that} \ \gamma_x(t)=t^n\cdot x$$
and the map $Z\times_X \widehat Z \to \cL(0, m)$ is the induced map on the correspondence spaces. 
Then it is not hard to see that $\iota_{n, \widehat Z}^*\cL^\rv_{\widehat H}\cong \hat\pi^*(\xi^{\otimes n})$ and $\iota_{n, \widehat Z}^* \nabla^{\cL^\rv_{\widehat H}}=\hat{\pi}^*\nabla^{\xi^{\otimes n}}.$ Comparing (\ref{loopHori}) with (\ref{exoticHori}), and (\ref{ivloopHori}) with (\ref{ivexoticHori}), we can see that for $\omega\in \gB_{-n}^{\bullet}(\LDQ_0 Z), \widehat \theta\in \gA^{\rv, \bullet}(\LDQ_n\widehat Z)$,
\be \tau_n\circ i_{0, Z}^*(\omega)=\iota_{n, \widehat Z}^*\circ \tau_{0, n}(\omega), \ \ \ \ \widehat \sigma_n\circ\iota_{n, \widehat Z}^*(\widehat \theta)=i_{0, Z}^*\circ \widehat\sigma_{n,0}^*(\widehat \theta).\ee
This shows that the exotic Hori formulae are indeed the shadows of the loop Hori formulae. 

%%%%%%%%%%%%%%%%%%%%%%%%%%%%%%%%%%%%%%%%%%%

\end{document}